\documentclass[10pt,a4paper]{article}
\pdfoutput=1
\usepackage[margin=1in,footskip=0.25in]{geometry}

\usepackage{amsmath}
\usepackage{amssymb}
\usepackage{amsthm}
\usepackage{mathtools}
\usepackage{xparse}
\usepackage{cite}
\usepackage{hyperref}
\usepackage{cleveref}
\usepackage[X2,T1]{fontenc}

\newtheorem{claim}{Claim}

\newtheorem{defn}[claim]{Definition}
\newtheorem{lem}[claim]{Lemma}
\newtheorem{thm}[claim]{Theorem}
\newtheorem{cor}[claim]{Corollary}

\DeclareMathOperator{\spn}{span}

\DeclareMathOperator{\mat}{Mat}
\DeclareMathOperator{\codim}{codim}

\NewDocumentCommand \bern { O{k} }{
    \rho_{#1}
}

\NewDocumentCommand \calLp { O{n} O{\omega} }{
    {\mathcal L'}^{(#1)}_{#2}
}
\NewDocumentCommand \calL { O{n} O{\sigma^{-n}\omega} }{
    \mathcal L^{(#1)}_{#2}
}
\NewDocumentCommand \Lab { O{a} O{b} }{
    \mathcal L_{#1\rightarrow#2}
}

\newcommand{\ubar}[1]{\text{\b{$#1$}}}

\title{Constructing the Oseledets decomposition with subspace growth estimates}
\author{George Lee}
\begin{document}
\maketitle
\abstract{
The semi-invertible version of Oseledets' multiplicative ergodic theorem providing a decomposition of the underlying state space of a random linear dynamical system into fast and slow spaces is deduced for a strongly measurable cocycle on a separable Banach space.
This work represents a significantly simplified means of obtaining the result, using measurable growth estimates on subspaces for linear operators combined with a modified version of Kingman's subadditive ergodic theorem.
}
\section{Introduction}

The multiplicative ergodic theorem is a fundamental tool in the study of linear dynamical systems.
Given a Banach space $X$ write $\mathcal B(X)$ for the space of bounded linear operators on $X$.
Let $\Omega$ be a Lebesgue probability space.
Given an ergodic system $\sigma:\Omega\rightarrow\Omega$ and function $\mathcal L:\Omega\rightarrow\mathcal B(X)$, one may compose copies of $\mathcal L$ along orbits and investigate long term behaviour of any $x\in X$.
Such an $\mathcal L$ is referred to as a \textit{cocycle}, and is said to be \textit{forward-integrable} if $\log^+\Vert\mathcal L\Vert\in L^1(\Omega)$.
Iteratively set $\calL[0][\omega]=1_X:X\rightarrow X$, and $\calL[n][\omega]:=\mathcal L_{\sigma^{n-1}\omega}\circ\calL[n-1][\omega]$ for $n\in\mathbb N$.
The multiplicity property $\calL[m+n][\omega]=\calL[m][\sigma^n\omega]\circ\calL[n][\omega]$ is clear.
If $\mathcal L$ is required to be invertible then $\calL[n][\omega]=\mathcal L_{\sigma^n\omega}^{-1}\circ\cdots\circ\mathcal L_{\sigma^{-1}\omega}^{-1}$ for $n<0$ extends the cocycle to all $n\in\mathbb Z$.
If $\mathcal L^{-1}$ is forward-integrable then $\mathcal L$ is said to be \textit{backward-integrable}.
One may seek ways of describing $X$ in terms of long term behaviour of vectors under $\calL[n][\omega]$ as $n\rightarrow\infty$.
Given any normed space $V$ write $\mathbb S_V=\{x\in V:\Vert x\Vert=1\}$ and $\mathbb B_V=\{x\in V:\Vert x\Vert<1\}$.
Oseledets \cite{oseledets1968multiplicative} proved the following in 1965:
\begin{thm}
Let $\Omega$ be a Lebesgue probability space and $\sigma:\Omega\rightarrow\Omega$ be an invertible measure preserving transformation.
Let $\mathcal L:\Omega\rightarrow GL_d(\mathbb R)$ be forward and backward integrable, where $GL_d$ denotes the invertible $d-$dimensional matrices.
Then there are measurable numbers $\lambda_i(\omega),i\in\{1,\cdots,r_\omega\}$ and a direct sum decomposition of invariant measurable subspaces $\mathbb R^d=\bigoplus_iE_i(\omega)$ such that for $x\in\mathbb S_{E_i(\omega)}$,
$$
\lim_{n\rightarrow\pm\infty}\tfrac1n\log\Vert\calL[n][\omega]x\Vert=\pm\lambda_i(\omega).
$$
This limit converges uniformly in $x$.
\end{thm}
In a setting where all invertibility assumptions are dropped the conclusions are weaker.
The following flag decomposition is proven in the work of Raghunathan \cite{raghunathan1979proof}:
\begin{thm}
Let $\Omega$ be a Lebesgue probability space and $\sigma:\Omega\rightarrow\Omega$ be a (not necessarily invertible) measure preserving transformation.
Let $\mathcal L:\Omega\rightarrow\mat_d(\mathbb R)$ satisfy $\log^+\Vert\mathcal L\Vert\in L^1(\Omega)$ where $\mat_d$ denotes the $d$-dimensional matrices.
Then there are measurable numbers $\lambda_i(\omega),i\in\{1,\cdots,r_\omega\}$ and a flag of invariant measurable subspaces $\mathbb R^d=V_1(\omega)>V_2(\omega)>\cdots>V_{r_\omega}(\omega)$ such that for $x\in V_i(\omega)\setminus V_{i+1}(\omega)$,
$$
\lim_{n\rightarrow\infty}\tfrac1n\log\Vert\calL[n][\omega]x\Vert=\lambda_i(\omega).
$$
\end{thm}

These results were subsequently extended to more general settings: by Ruelle to compact operators on Hilbert spaces\cite{ruelle1979ergodic}, by Ma\~n\'e to compact operators on Banach spaces with additional continuity assumptions \cite{mane1983lyapounov} and by Thieullen to quasicompact operators \cite{thieullen1987fibres}.

The case where the underlying vector space is a separable Banach space was first presented by Lian and Lu \cite{lian2010lyapunov}.
Their monograph obtains the decomposition result assuming almost-everywhere injectivity of the cocycle in a separable Banach space.

The injectivity condition isn't necessary to obtain a decomposition: Froyland, Lloyd and Quas \cite{froyland2010semi} demonstrated that as long as $\sigma$ is invertible the space may still be written as a sum of fast and slow spaces.
That paper dealt with the finite dimensional case, but in \cite{gonzalez2011semi} Gonz\'alez-Tokman and Quas presented a generalisation to a separable Banach space.
However, there is some ambiguity to the statement and consequentially the proof - it is claimed that a measurable composition of $X$ exists without a precise discussion of what this means, which is crucial given that in the infinite dimensional setting no statement may be made of the slow spaces, as discussed in Horan's thesis \cite{horan2020spectral}.
Gonz\'alez-Tokman and Quas' subsequent shorter proof in the case where the dual $X^*=\mathcal B(X,\mathbb R)$ is separable used intuitive notions of volume growth that this work partially builds upon, but relied on less than fully rigorous references to cocycles whose domain varies depending on $\omega$.
Another example of this kind of emphasis on volume growth may be seen in the noninvertible result of Blumenthal in \cite{blumenthal2016vol} where a flag is obtained with no separability assumptions but with a much stronger uniform measurability condition for the cocycle.

Given measurable spaces $A$ and $B$ write $\mathcal M(A\rightarrow B)$ for the space of measurable functions from $A$ to $B$.
Let $\mathcal F$ be the Borel sigma algebra induced by the strong operator topology on $X$.
The space $\mathcal{SM}$ of strongly measurable functions consists of measurable maps with respect to this choice of sigma algebra:
$$
\mathcal{SM}\big(A\rightarrow\mathcal B(X))=\mathcal M(A\rightarrow(\mathcal B(X),\mathcal F)).
$$
Write $\mathcal G_kX=\{V\leq X:\dim V=k\}$ for the \textit{Grassmannian} of subspaces of dimension $k$.
Given $T\in\mathcal B(X)$ define the slowest growth of vectors in a given subspace under $T$:
$$
g(T,V)=g_T(V)=\inf_{x\in\mathbb S_V}\Vert Tx\Vert,\bern T=\sup_{V\in\mathcal G_kX}g_T(V).
$$
The $\rho_k$ are called Bernstein numbers in the work of Pietsch\cite{pietsch1974}, which gives an overview of similar kinds of statistics in Banach spaces.
Given a strongly measurable forward-integrable cocycle $\mathcal L$ ona Lebesgue probability space there are decreasing sequences $(\mu_i)_{i\in\mathbb N}$ and $\lambda_i$ of invariant functions
$$
\mu_k=\lim_{n\rightarrow\infty}\tfrac1n\rho_k\calL[n][\omega],\lambda_1=\mu_1,\lambda_{i+1}=\mu_{\inf\{t:\mu_t<\lambda_i\}}.
$$
While there are countably many $\mu_i$ there may only be finitely many $\lambda_i$s, referred to henceforth as the Lyapunov exponents.
The main result, an extension of the semi-inveritble result to separable Banach spaces, may now be stated:
\begin{thm}\label{mainres}
Let $(\Omega,\sigma,X,\mathcal L)$ be a strongly measurable forward-integrable random linear dynamical system with $(\Omega,\sigma)$ an ergodic invertible map on a Lebesgue probability space and having Lyapunov exponents $(\mu_i)_{i=1}^\infty$ and $(\lambda_i)_{i=1}^L$, where $1\leq L\leq\infty$.
Then for each $0\leq l<L$ there is a direct sum decomposition into  equivariant spaces $X=\big(\bigoplus_{i\leq l}E_i(\omega)\big)\oplus V_{l+1}(\omega)$ with the $E_l:\Omega\rightarrow\mathcal G_{m_l}X$ measurable, $m_l\in\mathbb N$ and having
$$
\tfrac1n\log\Vert\calL[n][\omega]\vert_{E_i(\omega)}\Vert,\tfrac1n\log\inf\{\Vert\calL[n][\omega]x\Vert:x\in\mathbb S_{E_i(\omega)}\}\rightarrow\lambda_i,
$$
and
$$
V_l(\omega)=\{x\in X:\lim_{n\rightarrow\infty}\tfrac1n\log\Vert\calL[n][\omega]x\Vert\leq\lambda_l\}.
$$
The projection $\Pi:X\rightarrow\bigoplus_{j<i}E_{j}$ parallel to $V_i$ is strongly measurable and \textit{tempered}, that is to say, $\lim_{n\rightarrow\infty}\tfrac1n\log\Vert\Pi_{\sigma^n\omega}\Vert=0$ almost surely.
There is a nontrivial decomposition, $L\geq2$, exactly when $\nu=\lim_{n\rightarrow\infty}\mu_n<\lambda_1$.
\end{thm}

The choice of construction method for the Lyapunov exponents is a fundamental aspect to each proof: in \cite{raghunathan1979proof} $\mu_i$ are written simultaneously using finite dimensional singular value decomposition, while by contrast in \cite{lian2010lyapunov} the first Lyapunov exponent is defined according to the asymptotic growth rate of $\Vert\calL[n][\omega]\Vert$ followed by the Lyapunov exponents being first described as the asymptotic growth rate for nonzero elements of the fast spaces in the statement of the thorem.
While singular value decomposition is no longer available in these contexts, some construction must be written down that may be viewed as reminiscent of a proof of the singular value decomposition of a transformation.
The work presented here explicitly relies on a notion of singular values for arbitrary elements of $\mathcal B(X)$.
Given a $T\in\mathcal B(X)$, sufficient conditions on these singular values for contracting fast growing regions of $\mathcal GX$ under the action $V\mapsto TV$ are derived.
A modified version of Kingman's subadditive theorem is established in reminiscent used to guarantee the asymptotic growth rates of the singular values of $\calL[n][\omega]$.

The results from past papers stated thus far have not required ergodicity.
Throughout the rest of this paper, ergodicity will be assumed, simplifying the classification of invariant functions, although all results as is typical may be formulated with this assumption dropped.
In this Banach space setting, the main theorem yields in a trichotomous classification of forward-integrable cocycles on separable Banach spaces: one of
\begin{itemize}
\item $\mathcal L$ fails to be quasicompact, with $\mu_i=\nu$ for all $i\in\mathbb N$ - no fast spaces may be detected
\item $\mathcal L$ is quasicompact, with finitely many finite dimensional fast spaces growing at rates $\lambda_1>\cdots>\lambda_L>\lambda_{L+1}=\nu$
\item $\mathcal L$ is quasicompact, with a countable sequence of finite dimensional fast spaces growing at rates $\lambda_1>\lambda_2>\cdots\rightarrow\nu$.
\end{itemize}
These three possibilities just correspond to situations where there are no, finitely many or countably many fast spaces.
The number $\nu$ is an alternative choice to the typical index of compactness $\kappa=\lim_{n\rightarrow\infty}\tfrac1n\Vert\calL[n][\omega]\Vert_c$ that proves simpler to work with in this proof.
In the appendix it is verified that these quantities are equal, so that the final result represents an extension of the result of Lian and Lu with the injectivity assumption dropped, avoiding measurability or domain concerns present in \cite{gonzalez2011semi} and \cite{gonzalez2015concise}.
In particular, the advantages of the current work are in its brevity, the preciseness of the conclusion and the use of an intuitive geometric perspective.

The author is indebted to Anthony Quas for his enthusiastic encouragement and guidance throughout.
\section{The Grassmannian}
When discussing subspaces of $X$, we may also consider the space of subspaces of $X$ they lie in:
\begin{defn}
Given vector spaces $U,V\leq X$ with $U\oplus V=X$, write $\Pi_{U\vert V}$ for the projection defined by $\Pi_{U\Vert V}(u+v)=u$ for any $u\in U$ and $v\in V$.
The Grassmannian is defined as the set
\begin{align*}
\mathcal GX=&\{V\text{ closed }\leq X:\text{ there exists a projection }X\xrightarrow\Pi V\text{ with }\Vert\Pi\Vert<\infty\}\\
=&\{V\text{ closed }\leq X:\text{ there exists }W\leq X\text{ with }\Vert\Pi_{V\Vert W}\Vert<\infty\}.
\end{align*}
$\mathcal GX$ may be metrised by any of a few equivalent choices of distance between spaces, such as the Hausdorff distance between unit spheres.
$\mathcal G_kX\subset\mathcal GX$ has already been defined.
Write $\mathcal G^kX=\{V\leq X:\codim V=k\}\subseteq\mathcal GX$.
\end{defn}
\begin{lem}\label{GXsep}
Suppose that $X$ is a Banach space.
Then
\begin{itemize}
\item
If $X$ is separable then $\mathcal G_kX$ is separable and complete.
\item
If $\Omega$ is a measurable space, $E\in\mathcal M(\Omega\rightarrow\mathcal G_kX)$ and $A\in\mathcal{SM}(\Omega\rightarrow\mathcal B(X))$ then the pointwise pushforward $\omega\mapsto A(\omega)E(\omega)$ is measurable.
\end{itemize}
\end{lem}
\begin{proof}
Completeness is established in section 2.1 of chapter IV of \cite{kato2013perturbation}.
The pushforward result is established in corollary B.13 of \cite{gonzalez2011semi}.
\end{proof}
\begin{lem}
If we take $B=\mathcal B(X)$ then we have the following characterisation of the space of strongly measurable functions:
$$
\mathcal{SM}\big(\Omega\rightarrow\mathcal B(X)\big)=\left\{\mathcal L\colon\Omega\rightarrow\mathcal B(X):\begin{array}{c}\text{for each }x\in X,\\(\omega\mapsto\mathcal L_\omega x)\in\mathcal M\big(\Omega\rightarrow X\big)\end{array}\right\}.
$$
\end{lem}
\begin{proof}
Lemma A.4 of \cite{gonzalez2011semi} checks the equivalence of these conditions.
\end{proof}
The following hold:
\begin{lem}\label{growth_ineq}
Let $T\in\mathcal B(X,Y)$ and $S\in\mathcal B(Y,Z)$.
\begin{itemize}
\item
$g_T(V)g_S(TV)\leq g_{S\circ T}(V)\leq\min\{\Vert T\vert_V\Vert g_S(TV),g_T(V)\Vert S\Vert\}$
\item
$\rho_k(T)\rho_k(S)\leq\rho_k(S\circ T)\leq\rho_k(T)\Vert S\Vert$
\end{itemize}
\end{lem}
We make use of the following straightforward construction:
\begin{lem}\label{lexstruct}
let $\Omega$ be a measurable space.
Suppose that $\{x_n\}_{n\in\mathbb N}\subseteq X$ for some measurable space $X$.
Then if we can write down a measurable set $S_n\subseteq\Omega$ for each $n\in\mathbb N$ such that $\bigcup_nS_n=\Omega$ then there is an associated measurable $N\colon\Omega\rightarrow\mathbb N$ given by $N_\omega=\inf\{n:\omega\in S_n\}$.
Further, the map $\omega\mapsto x_{N_\omega}$ is measurable.
\end{lem}
This principle is applied to check measurability in the context of spaces which are separable.
\section{Measurable statistics for linear operators}
\begin{lem}
Let $X$ be a normed space.
Then $g\colon\mathcal B(X)\times\mathcal G_kX\rightarrow[0,\infty)$ is continuous.
\end{lem}
\begin{proof}
Let $\epsilon>0$.
Here use the following metric inducing the product topology:
$$
d\big((S,V),(T,W)\big)=\Vert S-T\Vert+d(V,W).
$$
First note that $g$ is continuous at $(0,V)$ for any $V$ since $g(0,V)=0$ and if $d(V,W)+\Vert T-0\Vert<\epsilon$ then $g(T,W)<\epsilon$.
Otherwise, the projection onto the sphere $p_{\mathbb S}T=\tfrac1{\Vert T\Vert}T$ is continuous.
Let $\epsilon>0$.
Since $g(T,V)=\Vert T\Vert g(p_{\mathbb S}T,V)$ for nonzero $T$, it suffices to check $g$ is continuous on $\mathbb S_{\mathcal B(X)}\times\mathcal G_kX$.
Let
$$
(S,V),(T,W)\in\mathbb S_{\mathcal B(X)}\times\mathcal G_kX
$$
be such that
$$
d\big((S,V),(T,W)\big)=\Vert S-T\Vert+d(V,W)<\epsilon.
$$
Since $V$ and $W$ are finite dimensional their spheres are compact and we may choose an $x\in\mathbb S_V$ with $\Vert Sx\Vert=g(S,V)$.
Choose a $y\in\mathbb S_W$ with $\Vert x-y\Vert<\epsilon$.
Then
\begin{align*}
g(T,W)\leq&\Vert Ty\Vert\leq\Vert T-S\Vert+\Vert S\Vert\Vert(y-x)\Vert+\Vert Sx\Vert\\
<&\epsilon+\epsilon+g(S,V).
\end{align*}
By the same argument swapping $(S,V)$ and $(T,W)$ then, $g(S,V)<g(T,W)+2\epsilon$ and continuity (in fact, uniform continuity) on $\mathbb S_{\mathcal B(X)}\times\mathcal G_kX$ is clear.
Thus $g$ is continuous on $\mathcal B(X)\times\mathcal G_kX$ as required.
\end{proof}
\begin{cor}
The quantities $\rho_k:\mathcal B(X)\rightarrow\mathbb R$ are measurable functions.
\end{cor}
\begin{proof}
The $\rho_k$ may now be written as the supremum of the functions $\{T\mapsto g(T,V):V\in\mathcal G_kX\}$ which may be written as the supremum of a countable family since $\mathcal G_kX$ is separable.
\end{proof}
\begin{lem}{Grassmannian contraction estimates:}\label{pushclose}
Let $T\in\mathcal B(X)$ and $\Theta\in\left(\bern[k+1]T,\bern T\right)$.
Suppose that $V,W\in\mathcal G_kX$ are choices of fast growing spaces:
$$
g_T(V),g_T(W)>\Theta.
$$
Then
$$
d(TV,TW)<\frac{2}{1-\frac{\bern[k+1]T}\Theta}\frac{\bern[k+1]T}\Theta.
$$
In particular, if $\Theta>2\bern[k+1]T$ then
$$
d(TV,TW)<4\frac{\bern[k+1]T}\Theta.
$$
\end{lem}
\begin{proof}
If $V=W$ there is nothing to prove so assume otherwise.
Let $a\in V\setminus W$.
Then $W\oplus\spn\{a\}\in\mathcal G_{k+1}X$ so that there is some $a-b\in W\oplus\spn\{a\}$ with $b\in W$ and $g_T(a-b)\leq \bern[k+1]T$.
Set $c=a-b$.
By fastness $\Vert a\Vert\leq\frac{\Vert Ta\Vert}\Theta$ and $\Vert b\Vert\leq\frac{\Vert Tb\Vert}\Theta$.

\begin{align*}
\Vert Tc\Vert\leq\Vert c\Vert \bern[k+1]T&\leq(\Vert a\Vert+\Vert b\Vert)\bern[k+1]T\\
                                      &\leq\frac{\Vert Ta\Vert+\Vert Tb\Vert}\Theta \bern[k+1]T\\
                                      &\leq \frac{2\Vert Ta\Vert+\Vert Tc\Vert}\Theta \bern[k+1]T,
\end{align*}
So that
$$
\Vert Tc\Vert<\frac{2}{1-\frac{\bern[k+1]T}\Theta }\frac{\bern[k+1]T}\Theta \Vert Ta\Vert.
$$
Then since $a$ isn't in the kernel then $d(\frac{Ta}{\Vert Ta\Vert},TW)\leq\frac{\Vert Tc\Vert}{\Vert Ta\Vert}<\frac{2}{1-\frac{\bern[k+1]T}\Theta}\frac{\bern[k+1]T}\Theta $.
Since the choice of $a$ was arbitrary,
$$
d(TV,TW)<\frac{2}{1-\frac{\bern[k+1]T}\Theta }\frac{\bern[k+1]T}\Theta.
$$
\end{proof}
\begin{lem}
Let $k\in\mathbb N$.
Then there is a measurable $U:\mathcal G_kX\rightarrow\mathcal G_{k-1}X$ such that $U(V)<V$.
\end{lem}
\begin{proof}
Let $\{U_1,U_2,\cdots\}$ be dense in $\mathcal G_{k-1}X$.
Suppose $V\in\mathcal G_kX$.
Let $s_0(V)=1$ and
$$
s_i(V)=\inf\{s:\sup_{v\in\mathbb S_{U_s}}d(v,V)<2^{-i}\text{ and }d(U_s,U_{s_{i-1}(V)})<2^{1-i}\}\text{ for }i>0.
$$
In other words, $U_{s_i(V)}$ is the first member of the dense collection to fall in the nonempty open set
$$
B_{2^{1-i}}(U_{s_{i-1}})\cap\bigcup_{W\in\mathcal G_{k-1}X:W<V}B_{2^{-i}}(W).
$$
Then $U_{s_i}$ is Cauchy and so the pointwise limit of measurable functions $U_{s_i(V)}$ is convergent to some measurable $U(V)\in\mathcal G_{k-1}X$.
Further, $\sup_{v\in\mathbb S_{U(V)}}d(v,V)\leq 2^{-i}$ for any $i$, so $U(V)<V$.
\end{proof}
\begin{lem}
Let $k\in\mathbb N$.
Let $U\in\mathcal M(\mathcal G_kX\rightarrow\mathcal G_{k-1}X)$ with $U(V)<V$ as guaranteed by the previous lemma.
Then there is a measurable $b:\mathcal G_kX\rightarrow\mathbb S_X$ such that $U\oplus\spn\{b\}$ is the identity on $\mathcal  G_kX$.
\end{lem} 
\begin{proof}
Let $\{x_1,x_2,\cdots\}$ be dense in $\mathbb S_X$.
Let $t_0(V)=1$ and
$$
t_i(V)=\inf\{t:d(U(V)\oplus\spn\{x_t\},V)<2^{-t}\text{ and }\Vert x_t-x_{t_{i-1}}\Vert<2^{1-i}\}\text{ for }i>0.
$$
Again we get that $x_{t_i}$ converges pointwise to some measurable $b$ with the desired properties.
\end{proof}
Applying this result recursively we obtain the following:
\begin{lem}\label{measbase}
Let $X$ be a separable Banach space.
Then for every $k>0$ there is a measurable map $b:\mathcal G_kX\rightarrow\mathbb S_X^k$ such that for each $V\in\mathcal G_kX$,
$$
V=\spn\{b_i(V)\}_{i=1}^k.
$$
\end{lem}
\begin{lem}\label{Tsmallgrow}
Let $V\in\mathcal G^kX$ with bounded projection $\Pi:X\rightarrow V$.
Let $T\in\mathcal B(X)$.
Then $\rho_{k+l}(T)\leq\rho_l(T\circ\Pi)$.
\end{lem}
\begin{proof}
\begin{align*}
\rho_{k+l}(T)=&\sup_{U\in\mathcal G_{k+l}X}g_T(U)\leq\sup_{U\in\mathcal G_{k+l}X}g_T(U\cap V)\\
=&\sup_{U\in\mathcal G_{\leq k+l,\geq l}V}g_T(U)=\sup_{U\in\mathcal G_lV}g_{T|_V}(U)\\
=&\sup_{U\in\mathcal G_lV} g_{T\circ\Pi}(U)\leq\rho_l(T\circ\Pi).
\end{align*}
\end{proof}
\begin{lem}\label{Tbiggrow}
Let $X=E\oplus V=E'\oplus V'$ with corresponding projections $\Pi:X\rightarrow V\in\mathcal G^kX$ and $\Pi':X\rightarrow V'\in\mathcal G^kX$.
Let $T\in\mathcal B(X)$ such that $\Pi'\circ T=T\circ\Pi$, and suppose further that $g_T(E)>\Vert T|_V\Vert$.
Then $\rho_l(T\circ\Pi)\leq4\Vert\Pi\Vert\Pi'\Vert\rho_{l+k}(T)$.
\end{lem}
\begin{proof}
Make use of the inequality $\rho_l(T\circ\Pi)\leq\Vert\Pi\Vert\rho_l(T|_V)$.
Let $U\in\mathcal G_lV$ with $g_{T\circ\Pi}(U)>e^{-\epsilon}\rho_l(\Pi\circ T|_V)=e^{-\epsilon}\rho_l(T|_V)$.
Let $x\in\mathbb S_{U\oplus E}$.
Observe that $\Vert\Pi'Tx\Vert\leq\Vert\Pi'\Vert\Vert Tx\Vert$ and similar for $1-\Pi'$ so that
\begin{align*}
\Vert Tx\Vert\geq&\max\{\tfrac{\Vert\Pi'T x\Vert}{\Vert\Pi'\Vert},\tfrac{\Vert(1-\Pi')Tx\Vert}{\Vert1-\Pi'\Vert}\}\geq\max\{\tfrac{g_{T\Pi}(U)\Vert\Pi x\Vert}{\Vert\Pi'},\tfrac{g_T(E)\Vert(1-\Pi)x\Vert}{\Vert1-\Pi'\Vert}\}\\
\geq&\tfrac1{1+\Vert\Pi'\Vert}\max\{g_{T\Pi}(U)\Vert\Pi x\Vert,g_T(E)\Vert1-\Pi x\Vert\}\\
\geq&\tfrac{e^{-\epsilon}\rho_l(T|_V)}{1+\Vert\Pi'\Vert}\max\{\Vert\Pi x\Vert,\Vert(1-\Pi)x\Vert\}\geq\frac{e^{-\epsilon}\rho_l(T|_V)}{4\Vert\Pi'\Vert}\geq\tfrac{e^{-\epsilon}\rho_l(T\circ\Pi)}{4\Vert\Pi'\Vert\Vert\Pi\Vert}.
\end{align*}
We may then conclude that $g_T(U)\geq\tfrac{e^{-\epsilon}\rho_l(T\circ\Pi)}{4\Vert\Pi'\Vert\Vert\Pi\Vert}$ with $\epsilon>0$ arbitrary, so that $\rho_l(T\circ\Pi)\leq4\rho_{l+k}(T\circ\Pi)\Vert\Pi'\Vert\Vert\Pi\Vert$ as required.
\end{proof}
\section{A balanced subadditive ergodic theorem}
Consider a decomposition of the interval $[a,b)=[a,c)\sqcup[c,b)\subset\mathbb Z$.
In discrete time, where these intervals represent different intervals in which dynamics may occur.
Given a cocycle $\mathcal L:\Omega\rightarrow\mathcal B(X)$, for fixed $\omega$ and each $b\geq a\in\mathbb Z$ we may define the map $\Lab(\omega):=\calL[b-a][\sigma^a\omega]$, which may be thought of as the evolution rule for $X$ from time $a$ to time $b$.
Under this notation it is easy to see that for any $a<c<b\in Z$, $\Lab[c][b]\circ\Lab[a][b]=\Lab$.
The inequality $\Vert\Lab\Vert\leq\Vert\Lab[a][c]\Vert\Vert\Lab[c][b]\Vert$ holds.
Taking $\log$ of each side, one obtains the following triangle inequality-like bound on the growth of points in $X$ from time $a$ to $b$:
$$
f_{a\rightarrow b}(\omega):=\log\Vert\Lab\Vert\leq\log\Vert\Lab[a][c]\Vert+\log\Vert\Lab[c][b]\Vert=f_{a\rightarrow c}(\omega)+f_{c\rightarrow b}(\omega).
$$
\begin{defn}
Let $\mathcal F=\{f_n\}_{n\in\mathbb N_0}\subseteq\mathcal M(\Omega\rightarrow\mathbb R)$ satisfy
$$
f_{m+n}(\omega)\leq f_m(\sigma^n\omega)+f_n(\omega).
$$
Then $\mathcal F$ is referred to as a subadditive family of measurable functions.
\end{defn}
The following view is useful:
\begin{defn}
A subadditive family $\{f_n\}_{n\in\mathbb N_0}$ generates a stationary subadditive process $\{f_{a\rightarrow b}:a<b,a,b\in\mathbb Z\}$ and vice versa via the relation
$$
f_{a\rightarrow b}:=f_{b-a}\circ\sigma^a.
$$
A subadditive process is a collection $\{f_{a\rightarrow b}\}_{a<b\in\mathbb Z}\subseteq\mathcal M(\Omega\rightarrow\mathbb R)$ such that for all $a<c<b\in\mathbb Z$,
$$
f_{a\rightarrow b}\leq f_{a\rightarrow c}+f_{c\rightarrow b},
$$
and the stationarity condition is 
$$
f_{a+l\rightarrow b+l}=f_{a\rightarrow b}\circ\sigma^l.
$$
As such both notations may be interchanged as appropriate.
A similar formalism is outlined in \cite{krengel2011ergodic}.
\end{defn}

The Kingman theorem concerns subadditive families of functions, here a slight refinement when the underlying transformation is invertible will be required.
\begin{thm}[Kingman\cite{kingman1968ergodic}]\label{ing}Let $(\Omega,\sigma,\mathbb P)$ be an ergodic system and let $\{f_n\}_{n\in\mathbb N}\subset L^1\Omega$ be subadditive.
Then there is a constant $C\in[-\infty,\infty)$ such that
$$
\tfrac1nf_n(\omega)\rightarrow C=\lim_{n\rightarrow\infty}\tfrac1n\int f_n
$$
pointwise almost everywhere.
\end{thm}
In the case where $\sigma$ is invertible, one easily obtains the following useful corollary:
\begin{cor}
Let $(\Omega,\sigma,\mathbb P)$ be an invertible ergodic system on a Lebesgue probability space and let $\{f_n\}_{n\in\mathbb N}\subset L^1\Omega$ be subadditive.
Then
$$
\lim_{n\rightarrow\infty}\tfrac1nf_n(\sigma^{-n}\omega)=\lim_{n\rightarrow\infty}\tfrac1nf_n(\omega)
$$
pointwise almost everywhere.
\end{cor}
\begin{proof}
Simply set $g_n=f_n\circ\sigma^{-n}$.
Then $g_n$ is subadditive with respect to $\sigma^{-1}$:
\begin{align*}
g_{m+n}(\omega)=f_{m+n}(\sigma^{-(m+n)}\omega)\leq&f_m(\sigma^{-(m+n)}\omega)+f_n(\sigma^{m-(m+n)}\omega)\\
=& g_m(\sigma^{-n}\omega)+g_n(\omega),
\end{align*}
so that certainly there exists some $L$ such that $\tfrac1nf_n\circ\sigma^{-n}\rightarrow L$.
It remains to check that this limit and $\tfrac1nf_n(\omega)\rightarrow L'$ coincide.
Let $\epsilon>0$.
Then there exists an $N$ such that
$$
\mathbb P(\omega:\tfrac1nf_n(\omega)\in(L'-\epsilon,L'+\epsilon)>\tfrac12
$$
and
$$
\mathbb P(\omega:\tfrac1ng_n(\omega)=\tfrac1nf_n(\sigma^{-n}\omega)\in(L-\epsilon,L+\epsilon))>\tfrac12,
$$
so that the set 
$$
\{\omega:\tfrac1nf_n(\sigma^n\omega)\in(L-\epsilon,L+\epsilon)\}\cap\{\omega:\tfrac1nf_n(\omega)\in(L'-\epsilon,L'+\epsilon)\}
$$
has positive measure, whence $|L-L'|<2\epsilon$.
$\epsilon$ was arbitrary though, so that $L=L'$.
\end{proof}
Kingman's original result may be modified in a few directions.
One may very rapidly obtain the following, which may be found in \cite{ferrando1995moving}:
\begin{lem}\label{easyking}
Let $\{f_n\}_{n\in\mathbb N}\subseteq L^1\Omega$ be a subadditive family of functions.
Then
$$
\tfrac1nf_{n\rightarrow2n}\rightarrow C=\lim_{n\rightarrow\infty}\tfrac1n\int f_n.
$$
\end{lem}
\begin{proof}
The lower bound is immediately clear, since
\begin{align*}
\tfrac1nf_{n\rightarrow2n}(\omega)\geq\tfrac1n(f_{0\rightarrow2n}-f_{0\rightarrow n})\rightarrow 2C-C=C\text{ as }n\rightarrow\infty.
\end{align*}
On the other hand, let $\epsilon>0$, then since $\tfrac1n\int f_n\rightarrow C$ there exists some $N$ such that for all $n\geq N$, $\tfrac1n\int f_n\leq C+\epsilon$.
Given some fixed $n\geq N$ and some $j<N$, set $j=\lfloor\tfrac nN\rfloor$ and break up the interval $[0,n)=[n,n+j)\sqcup[n+j,n+j+N),\cdots\sqcup[n+j+(t-1)N,n+j+\lfloor\tfrac nN\rfloor N)\sqcup[n+j+\lfloor\tfrac nN\rfloor N,2n)$.
Applying subadditivity in this manner
\begin{align*}
f_{n\rightarrow 2n}\leq&f_{n\rightarrow n+j}+\sum_{i=0}^t f_{n+j+iN\rightarrow n+j+(i+1)N}(\omega)+f_{n+j+tN\rightarrow n}\\
\leq& S_N|f_1|\circ\sigma^n(\omega)+\sum_{i=0}^tf_N\circ\sigma^{n+j+iN}(\omega)+S_N|f_1|\circ\sigma^{2n-N}(\omega).
\end{align*}
The above is valid for all $j\in[0,N)$, so averaging and dividing by $nN$ we obtain
\begin{align*}
\tfrac1nf_{n\rightarrow 2n}\leq&\tfrac1n(S_N|f_1|\circ\sigma^n+S_N|f_1|\circ\sigma^{2n-N})+\tfrac1{nN}S_nf_N\circ\sigma^n\\
\leq&2\epsilon+\tfrac 1{Nn}n(N(C+\epsilon))
\end{align*}
for sufficiently large $n$.
Since $\epsilon$ was arbitrary, the result is proven.
\end{proof}
A similar result for $f_{-n\rightarrow n}$ will be crucial to the main result.
The proof here uses the following lemma, which is a simplified version of result 3.9 in \cite{rudolph1990fundamentals}:
\begin{lem}[Backward Vitali]\label{bvit}
Let $(\Omega,\mathbb P,\sigma)$ be an invertible ergodic system on a Lebesgue probability space.
Suppose that there is a sequence of integer valued functions
$j_k\in\mathcal M(\Omega\rightarrow\mathbb N_0)$
such that $j_k(\omega)\rightarrow\infty$ as $k\rightarrow\infty$.
Then there is a measurable $A\subseteq\Omega$ and a measurable $j\in\mathcal M(A\rightarrow\mathbb N_0)$ such that writing $I_\omega=\{\sigma^i\omega:i\in\{-j_\omega,\cdots,j_\omega\}\}$,
the $I_\omega$ are disjoint and
$$
\mathbb P(\bigcup_{\omega\in A}I_\omega)>1-\epsilon.
$$
\end{lem}
\begin{thm}[Kingman's theorem for balanced intervals]\label{balancedkingman} let $(\Omega,\sigma,\mathbb P)$ be an ergodic system with invertible base, and let $(f_n)_{n\in\mathbb N}$ be  subadditive sequence of measurable functions on $\Omega$.
Then
$$
\tfrac1{2n}f_{2n}(\sigma^{-n}\omega)\rightarrow C=\lim_{n\rightarrow\infty}\tfrac1n\int f_n
$$
pointwise almost everywhere.
\end{thm}
\begin{proof}
Kingman's theorem immediately yields an upper bound: since
$$
\tfrac1{2n}f_{2n}(\sigma^{-n}\omega)\leq\tfrac1{2n}\big(f_n(\sigma^{-n}\omega)+f_n(\omega)\big)\rightarrow \tfrac12C+\tfrac12C
$$
it is clear that
$$
\bar C_\omega=\limsup_{n\rightarrow\infty}\tfrac1{2n}f_{2n}(\sigma^{-n}\omega)\leq C.
$$
If $C=-\infty$ then there is nothing more to show.
Otherwise,
\begin{align*}
\ubar C_{\sigma^{-1}\omega}=&\liminf_{n\rightarrow\infty}\tfrac1{2n}f_{2n}\circ\sigma^{-n}(\sigma^{-1}\omega)\\
\geq&\liminf_{n\rightarrow\infty}(\tfrac1{2n}(f_{2(n+1)}\circ\sigma^{-(n+1)}(\omega)-f_2\circ\sigma^{n-1}(\omega))\\
\geq&\liminf\tfrac1{2n}f_{2n}\circ\sigma^{-n}(\omega)-\limsup f_2\circ\sigma^{n-1}(\omega)\\
=&\ubar C_\omega-0,
\end{align*}
whence $\ubar C\leq\bar C$ is constant.
Fix $\epsilon>0$.
It is sufficient to show that $C-\epsilon\leq\Theta(\ubar C,\epsilon)$ where $\Theta(\ubar C,\epsilon)\rightarrow\ubar C$ as $\epsilon\rightarrow0$.
The set of $j$ such that $f_{-j\rightarrow j}$ is small is infinite, so denote the ordered elements
$$
\{j:f_{-j\rightarrow j}(\omega)<2j(\ubar C+\epsilon)\}=\{j_k(\omega):k\in\mathbb N\}.
$$
Since $f$ is integrable, it is uniformly integrable and so there exists a $\delta>0$ such that if $\mathbb P(A)<\delta$ then $\int_A f<\epsilon$.
Apply Backward Vitali to the $j_i$s and the $-j_i$s with parameter $\epsilon':=\min\{\epsilon,\delta\}$ to obtain an $A\subseteq\Omega$ and a $j:A\rightarrow\mathbb N$ according to theorem \ref{bvit} such that writing $B=\Omega\setminus\bigcup_{\omega\in A}I_\omega$, we have $\mathbb P(B)<\delta$ and so
$$
\int_{\Omega\setminus\bigcup_{\omega\in A}I_\omega}f_1^+<\epsilon.
$$
While every $j(\omega)=j_i(\omega)$, for some $i$ so that for every $\omega\in A$,
$$
f_{-j(\omega)\rightarrow j(\omega)}(\omega)\leq2j(\ubar C+\epsilon).
$$
Write $\Lambda_{\pm}=\{\sigma^{\pm j_i(\omega)}(\omega):\omega\in A\}$.
In addition we may define measurable maps $T_\pm\in\mathcal M(\Omega\rightarrow\mathbb N)$ by
\begin{gather*}
T_s(\omega)=\inf\{t\in\mathbb N_0:\sigma^t\omega\in\Lambda_-\},\\
T_e(\omega)=\inf\{t>T_s(\omega)\in\mathbb N_0:\sigma^{t}\omega\in\Lambda_+\}.
\end{gather*}
$T_s$ is the earliest nonnegative time such that $\sigma^{T_s(\omega)}\omega$ is the start of a period at growth rate guaranteed close to $\ubar C$, $T_e$ by contrast seeks the first time in the past which was the end of such an interval.
Each are almost surely finite since $\mathbb P(\Lambda_\pm)=\mathbb P(A)>0$.

The idea of the next step is that $[0,n)$ may be measurably broken up into intervals of slow growth and with small gaps inbetween.
Set $b_0=0$ and recursively iterate along the orbit for $i>0$: \begin{gather*}
a_i(\omega)=T_s(\sigma^{b_{i-1}(\omega)}\omega),\\
b_i(\omega)=T_e(\sigma^{b_{i-1}(\omega)}\omega).
\end{gather*}
For any $m\in\mathbb N_0$ we may set $t_m(\omega)=\max\{t:b_t(\omega)\leq m\}$.
Since $j$ and $T_\pm$ are measurable and the $t_m(\omega)\rightarrow\infty$ as $m\rightarrow\infty$, it is possible to pick an $N$ such that for all $n\geq N$, $\mathbb P(j\geq N)<\tfrac14\epsilon'$, $\mathbb P(t_n=0)<\tfrac14\epsilon'$ and $\mathbb P(T_\pm\geq N)<\tfrac14\epsilon'$.
Let $n\geq N$.
The orbit of $\omega$ is considered over the following intervals of times:
\begin{align*}
[0,n]=&[0,a_1(\omega))\cup[a_1(\omega),b_1(\omega))\cup\\
&[b_1(\omega),a_2(\omega))\cup[a_2(\omega),b_2(\omega)]\cup\\
&\cdots\\
&[b_{t_n(\omega)-1}(\omega),a_{t_n(\omega)}(\omega))\cup[a_{t_n(\omega)}(\omega),b_{t_n(\omega)}(\omega))\cup\\
&[b_{t_n(\omega)}(\omega),n).
\end{align*}
Time $a_i\rightarrow b_i$ is guaranteed to have low growth:
$$
f_{a_i(\omega)\rightarrow b_i(\omega)}(\omega)\leq(b_i(\omega)-a_i(\omega))(\ubar C_\omega+\epsilon),
$$
and in addition there is a good chance that the gaps between the $a_i$s and the $b_i$s won't be too large:
$$
\mathbb P(a_1(\omega)\geq N)=\mathbb P(T(\omega)\geq N)<\tfrac14\epsilon'.
$$
Since $T_-$ seeks the endpoints of the periods of guaranteed growth, it holds that either $T_-(\sigma^n\omega)=n-b_{t_n(\omega)}(\omega)$ or $t_n(\omega)=0$.
Therefore,
\begin{align*}
\mathbb P(n-b_{t_n(\omega)}(\omega)\geq N)\leq&\mathbb P(T_-(\sigma^n\omega)\geq N)+\mathbb P(t_n(\omega)=0)\\
<&\mathbb P(T_-\geq N)+\tfrac14\epsilon'<\tfrac12\epsilon'.
\end{align*}
Set
$$
\Lambda=\{a_1\geq N\text{ or }n-b_{t_n}(\omega)\geq N\},
$$
so that it is immediate from above that $\mathbb P(\Lambda)<\epsilon'$.
Given some measurable $g$ write $S_pg=\sum_{i=0}^{p-1}g\circ\sigma^i$.
We then repeatedly apply subadditivity to $f_{0\rightarrow n}$:
\begin{align*}
f_n(\omega)=&f_{0\rightarrow n}(\omega)\\
\leq&(f_{0\rightarrow a_1(\omega)}+f_{a_1(\omega)\rightarrow b_1(\omega)}+\cdots+f_{a_t(\omega)\rightarrow b_{t_{n\omega}}(\omega)}+f_{b_{t_{n\omega}}(\omega)\rightarrow n})(\omega)\\
\leq&\sum_{i=1}^{t_n(\omega)}f_{a_i(\omega)\rightarrow b_i(\omega)}(\omega)+\sum_{i=1}^{t_n(\omega)}f_{b_{i-1}(\omega)\rightarrow a_i(\omega)}(\omega)+f_{b_{t_n(\omega)}\rightarrow n}(\omega)\\
\leq&\sum_{i=1}^{t_n(\omega)}(b_i(\omega)-a_i(\omega))(\ubar C+\epsilon)+\\
&\sum_{i=a_1(\omega)}^{b_{t_n(\omega)}(\omega)}1_B(\sigma^i\omega)f_1\circ\sigma^i+f_{b_{t_n(\omega)}(\omega)\rightarrow n}(\omega).
\end{align*}
In the case where $t_n(\omega)=0$ the first two terms are empty sums.
Outside $\Lambda$ it is guaranteed that $[0,a_1(\omega))\subseteq[0,M)$ and $[b_{t_n(\omega)}(\omega),n)\subseteq [n-M,n)$.
Therefore
\begin{align*}
f_n(\omega)\leq&\sum_{i=0}^{n-1}(\ubar C+\epsilon)1_{B^c}(\sigma^i\omega)+\sum_{i=0}^{n-1}1_B(\sigma^i\omega)f_1(\sigma^i\omega)+\\
&\sum_{i=0}^M(|f_1|(\sigma^i\omega)+|\ubar C|+\epsilon)+\sum_{i=n-M}^{n-1}(|f_1|(\sigma^i\omega)+|\ubar C|+\epsilon)+\\
&1_\Lambda(\omega)S_n|f_1|(\omega)\\
=&(\ubar C+\epsilon)S_n1_{B^c}(\omega)+S_n(f_11_B)(\omega)+\\
&S_M|f_1|(\omega)+S_M|f_1\circ\sigma^{n-M}|(\omega)+1_\Lambda(\omega)S_n|f_1|(\omega).
\end{align*}
It remains to check that each of these terms are small enough to yield the result when we divide through by $n$ and integrate:
\begin{align*}
\tfrac1n\int(\ubar C+\epsilon)S_n1_{B^c}(\omega)+S_n(1_Bf_1)(\omega)=&\int_\Omega\big((\ubar C+\epsilon)1_{B^c}+f_11_B\big)\\
=&(\ubar C+\epsilon)\mathbb P(B^c)+\int_B|f_1|\\
\leq&\max\{(\ubar C+\epsilon)(1-\epsilon),\ubar C+\epsilon\}+\epsilon\\
\leq&(\ubar C+\epsilon)+\epsilon(|\ubar C|+\epsilon)+\epsilon,
\end{align*}
and
\begin{align*}
&\tfrac1n\int\big(S_M|f_1|(\omega)+S_M|f_1\circ\sigma^{n-M}|(\omega)+1_\Lambda(\omega)S_n|f_1|(\omega)\big)\\
=&\tfrac1n\int\big(M|f_1|+M|f_1|\big)+\tfrac1n\int_\Lambda S_n|f_1|\\
\leq&\frac{2M\epsilon}n+\tfrac1n\sum_{i=0}^{n-1}\int_{\sigma^i\Lambda}|f_1|\\
\leq&\frac{2M\epsilon}n+\epsilon.
\end{align*}
Putting these together we obtain
\begin{align*}
(C-\epsilon)\leq\tfrac1n\int_\Omega f_n\leq&(\ubar C+\epsilon)+\epsilon(|\ubar C|+\epsilon)+\epsilon+\frac{2M\epsilon}n+\epsilon\\
\rightarrow&\ubar C+(3+|\ubar C|)\epsilon\text{ as }n\rightarrow\infty.
\end{align*}

Letting $\epsilon\rightarrow0$ we obtain $C\leq\ubar C$ as required.
\end{proof}
Linearity of the cocycles yields subadditivity of these families, which allow us to apply Kingman.
\begin{defn}
Given a random linear dynamical system $\mathcal R=(\Omega,\sigma,X,\mathcal L)$, write
$$
\lambda_\omega(x)=\limsup_{n}\tfrac1n\log\Vert\mathcal L^{(n)}_\omega x\Vert\in[-\infty,\infty).
$$
Write $m_i$ for the multiplicity of $\lambda_i$, ie, $m_i$ is defined as the largest $m\in\mathbb N$ such that
$$
\mu_{m_1+\cdots+m_{i-1}+m}=\lambda_i.
$$
\end{defn}
\begin{lem}
The quantities $\lambda_i,\mu_i$ and $\nu$ are almost everywhere constants.
Further, $\lambda_2$ exists if and only if $\lambda_1>\nu$.
\end{lem}
\begin{proof}
The existence of the limits are guaranteed by applying Kingman to particular choices of subadditive sequences:
\begin{gather*}
f^k_n(\omega)=\log\bern(\calL[n][\omega]),\\
g_n(\omega)=\log\Vert\calL[n][\omega]\Vert_c.
\end{gather*}
The second statement is a trivial consequence of the definitions $\nu=\lim_{n\rightarrow\infty}\mu_n$ and $\lambda_2=\mu_{\inf\{t:\mu_t<\mu_1\}}$.
\end{proof}
\section{Decomposing a cocycle}
The tools obtained thus far are now used to decompose quasicompact cocycles.
\begin{lem}
Let $\mathcal R=(\Omega,\sigma,X,\mathcal L)$ be a strongly measurable random linear dynamical system with ergodic base on a separable Banach space.
Suppose that $\mathcal L$ satisfies the quasicompactness condition $\nu<\lambda_1$.
Then there is a unique measurable choice of fast space $E:\Omega\rightarrow\mathcal G_{m_1}X$ for which the following hold almost surely:
\begin{itemize}
\item equivariance: $\mathcal L_\omega E(\omega)=E(\sigma\omega)$
\item $\lambda(x)=\lambda_1$ for every $x\in E(\omega)\setminus\{0\}$
\item $\lim_{n\rightarrow\infty}\tfrac1n\log g(\calL[n][\omega],E(\omega))=\lambda_1.$
\end{itemize}
\end{lem}
\begin{proof}
Since $X$ is separable, by \cref{GXsep} $\mathcal G_{m_1}X$ is separable: choose some dense $\{E_i\}_{i\in\mathbb N}\subseteq\mathcal G_{m_1}X$.
Let $\epsilon\in(0,\tfrac15(\lambda_1-\lambda_2))$.
The sets
$$
A_i=\{\omega\in\Omega:g(\calL[2n],E_i)>e^{-\epsilon} \bern(\calL[2n]))\}
$$
are measurable and cover $\Omega$, since for fixed $\omega$, density of the $E_i$s and continuity of $g_{\calL[2n][\sigma^{-n}\omega]}$ means that we may find an $i$ with $g(\calL[2n][\sigma^{-n}\omega],E_i)$ as close to $\bern(\calL[2n][\sigma^{-n}\omega])$ as we please.
Then \cref{lexstruct} above provides measurable functions
\begin{gather*}
\iota^{(n)}(\omega)=\inf\{i\in\mathbb N:g(\calL[2n],E_i)>e^{-\epsilon}\bern (\calL[2n])\},\\
\tilde E^{(n)}(\omega):=E_{\iota^{(n)}\omega},
\end{gather*}
and the pushforward
\begin{gather*}
E^{(n)}(\omega)=\calL\tilde E^{(n)}(\omega)
\end{gather*}
is also then measurable by \cref{GXsep}.

First, we establish that this sequence is Cauchy, and therefore convergent, to a family of spaces $E\in\mathcal M(\Omega\rightarrow\mathcal G_{m_1}X)$.

For almost every $\omega\in\Omega$ the fastest $m_1$ dimensional growth rate
\begin{gather*}
\tfrac1n\log\bern(\calL)\rightarrow\lambda_1,\\
\tfrac1n\log\bern(\calL[n][\omega])\rightarrow\lambda_1,\\
\tfrac1{2n}\log\bern(\calL[2n])\rightarrow\lambda_1,\\
\end{gather*}
and
$$
\bern (\calL)\geq g_{\calL}(\tilde E^{(n)}(\omega))>e^{-\epsilon}\bern (\calL).
$$
Thus for each $\omega$ in this full measure set we may choose $M_\omega$ such that for $n\geq M_\omega$ we can usefully estimate growth under $\mathcal L$ in a few cases:
\begin{gather*}
\Vert\calL\Vert,\Vert\calL[n][\omega]\Vert\in(e^{n(\lambda_1-\epsilon)},e^{n(\lambda_1+\epsilon)}),\\
g(\calL[2n],\tilde E^{(n)}(\omega))\in(e^{2n(\lambda_1-\epsilon)},e^{2n(\lambda_1+\epsilon)})\\
\end{gather*}
and
$$
\Vert\mathcal L_{\sigma^{-(n+1)}\omega}\Vert<e^{n\epsilon}.
$$
Applying the inequalities in \cref{growth_ineq} then,
$$
g(\calL,\tilde E^{(n)}(\omega))\geq\frac{g(\calL[2n],\tilde E^{(n)}(\omega))}{\Vert\calL[n][\omega]\Vert}\geq e^{n(\lambda_1-3\epsilon)}
$$
and
$$
g(\calL[n][\omega],\calL\tilde E^{(n)}(\omega))=g(\calL[n][\omega],E^{(n)}(\omega))\geq\frac{g(\calL[2n],\tilde E^{(n)}(\omega))}{\Vert\calL\Vert}\geq e^{n(\lambda_1-3\epsilon)}.
$$
In addition, $\mathcal L_{\sigma^{-(n+1)}\omega}\tilde E^{(n+1)}(\omega)$ is also guaranteed to have fast growth under $\calL$:
\begin{align*}
g(\calL,\mathcal L_{\sigma^{-(n+1)}\omega}\tilde E^{(n+1)}(\omega)\geq&\frac{g(\calL[n+1][\sigma^{-(n+1)}\omega],\tilde E^{(n+1)}(\omega))}{\Vert\mathcal L_{\sigma^{-(n+1)}\omega}\Vert}\\
>&e^{(n+1)(\lambda_1-3\epsilon)}e^{-n\epsilon}\geq e^{n(\lambda_1-4\epsilon)}.
\end{align*} 
$E^{(n)}(\omega)$ consists then of the image of vectors that were fast from time $-n$ to $0$, and will grow fast from time $0$ to $n$.
Then by \cref{pushclose} with $\Theta=e^{n(\lambda_1-4\epsilon)}$ we have
\begin{align*}
d(E^{(n)}(\omega),E^{(n+1)}(\omega))=&d\left(\calL\tilde E^{(n)}(\omega),\calL(\mathcal L_{\sigma^{-(n+1)}\omega}\tilde E^{(n+1)}(\omega))\right)\\
<&4\frac{e^{n(\lambda_2+\epsilon)}}{e^{n(\lambda_1-4\epsilon)}}<e^{-n(\lambda_1-\lambda_2-5\epsilon)}.
\end{align*}
Thus $E^{(n)}(\omega)$ is Cauchy and convergent since $\mathcal G_{m_1}X$ is complete, say to $E(\omega)$.

To prove equivariance, observe that for $n\geq\max\{M_\omega,M_{\sigma\omega}\}$, we find $\tilde E^{(n+1)}(\sigma\omega)$ is fast under $\calL$:
$$
g(\calL,\tilde E^{(n+1)}(\sigma\omega))\geq\frac{g(\calL[n+1][\sigma^{-n}\omega],E^{n+1}(\sigma\omega))}{\Vert\mathcal L_\omega\Vert}\geq e^{(n+1)(\lambda_1-\epsilon)}e^{-n\epsilon}>e^{n(\lambda_1-2\epsilon)}.
$$
Then, once again, by closeness of images of fast spaces,
$$
d(\calL[n][\sigma^{-n}\omega]\tilde E^{(n+1)}(\sigma\omega),E^{(n)}(\omega))<e^{-n(\lambda_1-\lambda_2-3\epsilon)},
$$
whence
\begin{align*}
d(E^{(n+1)}(\sigma\omega),\mathcal L_\omega E^{(n)}(\omega))=&d(\calL[n+1][\sigma^{-(n+1)}\omega]\tilde E^{(n+1)}(\sigma\omega),\mathcal L_\omega E^{(n)}(\omega))\\
\leq&e^{n\epsilon}d(\calL[n][\sigma^{-(n+1)}\omega]\tilde E^{(n+1)}(\sigma\omega),E^{(n)}(\omega))\\
<&e^{-n(\lambda_1-\lambda_2-4\epsilon)},
\end{align*}
so that $\mathcal L_\omega E(\omega)=\lim_{n\rightarrow\infty}\mathcal L_\omega E^{(n)}(\omega)=E(\sigma\omega)$.

To check that $E(\omega)$ is fast, choose $x\in\mathbb S_{E(\omega)}$.
Then since for $n\geq M_\omega$ we have $d(E^{(n)}(\omega),E(\omega))<e^{-n(\lambda_1-\lambda_2-\epsilon)}$, for each such $n$ we may choose an $x_n\in\mathbb S_{E^{(n)}(\omega)}$ with $\Vert x-x_n\Vert<e^{-n(\lambda_1-\lambda_2-\epsilon)}$.
Since $x_n\in E^{(n)}(\omega)$ and $g(\calL[n][\omega],E^{(n)}(\omega)\geq e^{n(\lambda_1-3\epsilon)}$ it then follows that,
\begin{align*}
\Vert\mathcal L^{(n)}_\omega x\Vert\geq&\Vert\mathcal L^{(n)}_\omega x_n\Vert-\Vert\mathcal L^{(n)}_\omega(x-x_n)\Vert\\
                                   \geq&e^{n(\lambda_1-3\epsilon)}-e^{-n(\lambda_1-\lambda_2-\epsilon)}\Vert\mathcal L^{(n)}_\omega\Vert\\
                               \implies&\Vert\mathcal L^{(n)}_\omega x\Vert\geq\tfrac12 e^{n(\lambda_1-\epsilon)}.
\end{align*}
Thus we may conclude that as well as being equivariant, $E(\omega)$ is fast for all sufficiently large $n$; since the choice of $x$ was arbitrary the growth is uniform:
$$
g(\calL[n][\omega],E(\omega))>e^{n(\lambda_1-\epsilon)}.
$$
On the other hand, for $n$ sufficiently large we also have
$$
g(\calL[n][\omega],E(\omega))\leq \bern (\calL[m][\omega])<e^{n(\lambda_1+\epsilon)},
$$
whence $\tfrac1n\log g(\calL[n][\omega],E(\omega))\rightarrow\lambda_1$.
Finally, we check uniqueness:  Suppose that $E(\omega)$ and $E'(\omega)$ are both equivariant and fast, so that for every $\omega$ there is some $N$ such that for $n\geq N$,
$$
g_{\calL[n][\omega]}(E(\omega)),g_{\calL[n][\omega]}(E'(\omega))>e^{n(\lambda_1-\epsilon)}.
$$
Define $\varphi\in\mathcal M(\Omega\rightarrow[0,1])$ by  $\varphi(\omega)=d(E(\omega),E'(\omega))$.
Applying lemma \ref{pushclose}, for almost every $\omega$ we have
\begin{align*}
\varphi(\sigma^n\omega)=&d(E(\sigma^n\omega),E'(\sigma^n\omega))=d(\calL[n][\omega]E(\omega),\calL[n][\omega]E'(\omega))\\
<&4\frac{\bern[k+1]\calL[n][\omega]}{e^{n(\lambda_1-\epsilon)}}\rightarrow0.
\end{align*}
$\varphi$ tends to zero along all orbits.
Therefore the sets $\{\varphi(\omega)>\epsilon\}$ all have measure zero, whence $\varphi$ vanishes almost everywhere and $E=E'$.
\end{proof}
The following lemma provides the top fast space for a general quasicompact cocycle:
\begin{lem}\label{topdecomp}
Let $\mathcal R=(\Omega,\sigma,X,\mathcal L)$ be a quasicompact, semi-invertible random linear dynamical system.
Then there exists a forward-equivariant decomposition $X=E(\omega)\oplus V(\omega)$ , where $V:\Omega\rightarrow\mathcal G^kX$ and the corresponding projection is a strongly measurable $\Pi:\Omega\rightarrow\mathcal B(X)$.
$V$ is a slow growing space: $\lim_{n\rightarrow\infty}\tfrac1n\log\Vert\calL[n][\omega]\vert_{V(\omega)}\Vert\rightarrow\lambda_2$ almost surely.
Finally, $\Pi_\omega$ is tempered.
\end{lem}
\begin{proof}
By \cref{measbase} we may choose a measurable family of bases
$$
(b_i)_{i=1}^k:\mathcal G_{m_1}X\rightarrow\mathbb S_X^k.
$$
Write $v_i(\omega)=b_i(E(\omega))$ which is itself then measurable.
Let $\{(q_{ij})_{i=1}^k\}_{j\in\mathbb N}$ be dense in $\mathbb R^k$.
Let $T_j\in\mathcal M(\Omega\times X\rightarrow X)$ be defined by
$$
T_j(\omega,x)=x-\sum_{i=1}^kq_{ij}v_i(\omega),
$$
so that for each $\omega\in\Omega$, the collection $\{T(\omega,\cdot)\}$ is dense in translations of $X$ by elements of $E(\omega)$.
Let $\epsilon<\tfrac12(\lambda_1-\lambda_2)$.
Define $\iota_n\in\mathcal M(X\setminus E(\omega)\rightarrow\mathbb N)$ for $n\in\mathbb N$ by
$$
\iota_n(\omega,x)=\inf\{j\in\mathbb N:\Vert\mathcal L^{(n)}_\omega T_jx\Vert\leq e^\epsilon \bern[k+1](\mathcal L_\omega^{(n)})\Vert T_jx\Vert\}.
$$
$\Pi^{(n)}\in\mathcal M(X\rightarrow X)$ may then be defined piecewise by
\begin{gather*}
\Pi^{(n)}_\omega(x)=
\begin{cases}
0\text{, if }x\in E(\omega),\\
T_{\iota_n(\omega,x)}x\text{ otherwise,}
\end{cases}\text{ and}\\
P^{(n)}_\omega(x)=x-\Pi^{(n)}_\omega(x).
\end{gather*}
To see that $\iota$ is finite on $X\setminus E(\omega)$, first note that the set
$$
\{T_jx:j\in\mathbb N\}
$$
is dense in the set $S=E(\omega)+x$.
The set
$$
U=\{y\in E(\omega)\oplus\spn\{x\}:\Vert\calL[n][\omega]y\Vert<e^{\tfrac12\epsilon}\bern[k+1](\calL[n][\omega])\Vert y\Vert\}
$$
is open and scale invariant - for every $\theta\neq 0$ we have $\theta U=U$, so
$$
U'=(E(\omega)+x)\cap U
$$
is open and nonempty in $S$, whence 
$$
U'\cap\{T_jx\}_j\neq\emptyset,
$$
so that an $\iota_n(\omega,x)$ may be found in finite time.
To see that $\iota_n$ is measurable, note that
$$
\iota_n^{-1}\{1,\cdots,m\}=\bigcup_{i=1}^m\{(\omega,x):\Vert\calL[n][\omega]T_jx\Vert\leq e^\epsilon\bern[k+1](\calL[n][\omega])\Vert\calL[n][\omega]T_jx\Vert\}
$$
Let $x\in\mathbb S_X$.
Immediately from the definition, $P^{(n)}_\omega x\in E(\omega)$.
By convergence of the $\tfrac1n\log \bern $s, for all $n$ greater than or equal to some $N_\omega$,
\begin{gather*}
\Vert\mathcal L^{(n+1)}_\omega \Pi^{(n)}_\omega(x)\Vert\leq e^{(n+1)(\lambda_2+\epsilon)}\Vert\Pi^{(n)}_\omega x\Vert,\\
\Vert\mathcal L^{(n+1)}_\omega \Pi^{(n+1)}_\omega(x)\Vert\leq e^{(n+1)(\lambda_2+\epsilon)}\Vert\Pi^{(n)}_\omega x\Vert,\\
\Vert\mathcal L_{\sigma^n\omega}\Vert<e^{n\epsilon}\text{ and}\\
g_{\calL[n][\omega]}(E(\omega))\geq e^{n(\lambda_1-\epsilon)}.
\end{gather*}
In addition there is an easy bound independent of $x\in\mathbb S_X$ on $\Pi^{(n)}_\omega$:
\begin{align*}
\Vert\Pi^{(n)}_\omega x\Vert\leq&\Vert x\Vert+\Vert P^{(n)}_\omega x\Vert\\
\leq&1+e^{-n(\lambda_1-\epsilon)}\Vert\calL[n][\omega]P^{(N)}_\omega x\Vert\\
\leq&1+e^{-n(\lambda_1-\epsilon)}\big(\Vert\calL[n][\omega]x\Vert+\Vert\calL[n][\omega]\Pi^{(n)}_\omega x\Vert\big)\\
\leq&1+e^{-n(\lambda_1-\epsilon)}\big(e^{n(\lambda_1+\epsilon)}+e^{n(\lambda_2+\epsilon)}\Vert\Pi^{(n)}_\omega x\Vert\big),
\end{align*}
which rearranged yields
$$
\Vert\Pi^{(n)}_\omega x\Vert\leq\frac{1+e^{2n\epsilon}}{1-e^{-n(\lambda_1-\lambda_2-2\epsilon)}}\leq e^{3n\epsilon}.
$$
Consider differences between successive approximate slow components:
\begin{align*}
\Vert \Pi^{(n+1)}_\omega(x)-\Pi^{(n)}_\omega(x)\Vert=&\Vert P^{(n+1)}_\omega(x)-P^{(n)}_\omega(x)\Vert\\
\leq&\frac{\Vert\mathcal L^{(n+1)}_\omega(P^{(n+1)}_\omega x-P^{(n)}_\omega(x))\Vert}{g_{\calL[n][\omega]}(E(\omega))}\\
=&\frac{\Vert\mathcal L^{(n+1)}_\omega(\Pi^{(n+1)}_\omega x-\Pi^{(n)}_\omega(x))\Vert}{g_{\calL[n][\omega]}(E(\omega))}\\
\leq&e^{-(n+1)(\lambda_1-\epsilon)}\big(\Vert\calL[n+1][\omega] \Pi^{(n)}_\omega(x)\Vert+\Vert\calL[n+1][\omega] \Pi^{(n+1)}_\omega(x)\Vert\big)\\
\leq&e^{-(n+1)(\lambda_1-\epsilon)}\big(\Vert\mathcal L_{\sigma^n\omega}\Vert\Vert\calL[n][\omega] \Pi^{(n)}_\omega(x)\Vert+e^{n(\lambda_2+\epsilon)}\Vert\Pi^{(n+1)}_\omega(x)\Vert\big)\\
\leq&e^{-(n+1)(\lambda_1-\lambda_2-3\epsilon)}\big(e^{3n\epsilon}+e^{3(n+1)\epsilon}\big)\\
\leq&e^{-(n+1)(\lambda_1-\lambda_2-7\epsilon)},
\end{align*}
whence gaps between subsequent points decay exponentially, so that $(\Pi^{(n)}_\omega(x))_{n\in\mathbb N}$ forms a Cauchy and thus convergent sequence.
Set
$$
\Pi_\omega(x)=\lim_{n\rightarrow\infty}\Pi^{(n)}_\omega(x).
$$
Note that $\Pi_\omega\mathbb S_X$ is then bounded, since
\begin{align*}
\Vert\Pi_\omega x\Vert=&\lim_{m\rightarrow\infty}\Vert \Pi^{(m)}_\omega x\Vert\\
\leq&\Vert\Pi^{(N)}_\omega x\Vert+\lim_{m\rightarrow\infty}\Vert\Pi^{(N)}_\omega x-\Pi^{(m)}_\omega(x)\Vert\\
\leq&e^{3N\epsilon}+\frac{e^{-N(\lambda_1-\lambda_2-7\epsilon)}}{1-e^{-(\lambda_1-\lambda_2-7\epsilon)}},
\end{align*}
the final line being independent of choice of $x$.
Not only then is $\Pi^{(n)}_\omega x$ convergent, but we have the estimate
$$
\Vert\Pi^{(n)}_\omega x-\Pi_\omega x\Vert\leq\sum_{i=1}^\infty\Vert\Pi^{(n+i)}_\omega x-\Pi^{(n+i-1)}_\omega x\Vert<\frac{e^{-n(\lambda_1-\lambda_2-\epsilon)}}{1-e^{-\lambda_1-\lambda_2-\epsilon}}.
$$

$\Pi_\omega$ is linear: to see this let $b,c\in X$ and $t\in\mathbb R$, and set
$$
d:=\Pi_\omega^{(n)}(b+tc)-\Pi_\omega^{(n)}b-t\Pi_\omega^{(n)}c.
$$
Certainly $d\in E(\omega)$, since
\begin{align*}
d=&\Pi_\omega^{(n)}(b+tc)-(b+tc)+b-\Pi_\omega^{(n)}b+tc-t\Pi_\omega^{(n)}c\\
=&-P_\omega^{(n)}(b+tc)+P_\omega^{(n)}b+tP_\omega^{(n)}c\in E(\omega).
\end{align*}
Then applying the estimate for $\Pi^{(n)}_\omega$
\begin{align*}
\Vert d\Vert\leq&e^{-n(\lambda_1-\epsilon)}\Vert\calL[n][\omega]d\Vert\\
\leq&e^{-n(\lambda_1-\epsilon)}(\Vert\calL[n][\omega]\Pi^{(n)}_\omega(b+tc)\Vert+\Vert\calL[n][\omega]\Pi^{(n)}_\omega(b)\Vert+\Vert t\calL[n][\omega]\Pi^{(n)}_\omega c\Vert)\\
<&e^{-n(\lambda_1-\lambda_2-2\epsilon)}\big(\Vert\Pi^{(n)}_\omega(b+tc)\Vert+\Vert\Pi^{(n)}_\omega b\Vert+\Vert t\Pi^{(n)}_\omega c\Vert\big)\\
<&e^{-n(\lambda_1-\lambda_2-2\epsilon)}\big(\Vert\Pi_\omega(b+tc)\Vert+\Vert\Pi_\omega b\Vert+\Vert t\Pi_\omega c\Vert+\\
&(\Vert b+tc\Vert+\Vert b\Vert+\Vert tc\Vert)\frac{e^{-n(\lambda_1-\lambda_2-\epsilon)}}{1-e^{-\lambda_1-\lambda_2-\epsilon}}\big)\rightarrow0\text{ as }n\rightarrow\infty.
\end{align*}
Thus $d=0$ and $\Pi_\omega\in\mathcal B(X)$.
By construction $\Pi_\omega E(\omega)=0$.
On the other hand, $P_\omega X=E(\omega)$ since for any $x\in X$ and $n\in\mathbb N$ we have $x-\Pi^{(n)}_\omega x\in E(\omega)$.
$\Pi_\omega$ is then idempotent, since for any $x\in X$
$$
\Pi_\omega^2x-\Pi_\omega x=\Pi_\omega\circ P_\omega x\in\Pi_\omega E(\omega)=\{0\},
$$
and is thus a projection.
Set $V(\omega)=\Pi_\omega X$ so that $X=V(\omega)\oplus E(\omega)$.
For all $n\geq N_\omega$, and any $x\in V_\omega(x)$, because of the exponential rate of convergence there is a sequence of approximants
$$
\Vert x_n-x\Vert<C_\omega e^{-n(\lambda_1-\lambda_2-7\epsilon)}
$$
with
$$
\Vert\calL[n][\omega] x_n\Vert\in[0,e^{n(\lambda_2+\epsilon)}).
$$
Therefore,
$$
\Vert\calL[n][\omega]x\Vert<\Vert\calL\Vert\Vert x-x_n\Vert+\Vert\calL[n][\omega]x_n\Vert<e^{n(\lambda_2+8\epsilon)},
$$
whence
$$
\Vert\calL[n][\omega]\vert_{V(\omega)}\Vert<e^{n(\lambda_2+8\epsilon)}
$$
for $n\geq N_\omega$.
By the definition of $\lambda_2$, it is possible to choose $Y_n\in\mathcal G_{m_1+1}(\omega)$ with $\tfrac1ng_{\calL[n][\omega]}(Y_n)\rightarrow\lambda_2$, which means by dimension counting that there is always some $x_n\in\mathbb S_{V(\omega)\cap Y_n}$ with $\tfrac1n\Vert\calL[n][\omega]x_n\Vert\rightarrow\lambda_2$ and so
$$
\tfrac1n\log\Vert\calL[n][\omega]\vert_{V(\omega)}\Vert\rightarrow\lambda_2.
$$
Further, the map $\omega\mapsto\Pi_\omega$ is strongly measurable, since for each $x\in X$ the map $\omega\mapsto\Pi_\omega x$ is the limit of a sequence of measurable functions.

To see that $V(\omega)$ is equivariant it is sufficient to show that $P_{\sigma\omega}\circ\mathcal L_\omega x=0$ for all $x\in V(\omega)$.
If this were not the case, then $\mathcal L_\omega x$ would have a nonzero component in $E(\sigma\omega)$.
From this is would follow that $\lambda_\omega(x)=\lambda_1$ which would contradict the fact that $\Vert\calL[n][\omega]\vert_{V(\omega)}\Vert<e^{n(\lambda_2+8\epsilon)}$ for sufficiently large $n$.

As for temperedness of the projections:
Since $\Pi_\omega$ is bounded pointwise we may choose an $A\subseteq\Omega$ of positive measure on which $\Vert\Pi_\omega\Vert$ is at most some $M>0$.
Then define a new cocycle $\mathcal L'$ by
$$
\mathcal L_\omega'=\begin{cases}\mathcal L_\omega\circ\Pi_\omega\text{ if }\omega\in A,\\\mathcal L\text{ otherwise.}\end{cases}
$$
Then $\mathcal L_\omega'$ is forward-integrable since
$$
\int_\Omega\log^+\Vert\mathcal L_\omega'\Vert\leq \int_\Omega\log^+\Vert\mathcal L\Vert+\int_A\log^+M<\infty.
$$
Since $A$ has positive measure, there is almost surely an $n$ such that 
$\calLp E(\omega)=0$,
and so
$$
\calLp X\subseteq V(\sigma^n\omega).
$$
We may then conclude that for each $x\in X$, $\lambda_\omega'(x)\leq\lambda_2$ and $\lambda_1'\leq\lambda_2$.
Applying lemma \ref{easyking} to the subadditive families $g_n=\log\Vert\calL[n][\omega]\Vert$ and $g_n'=\log\Vert\calLp\Vert$ there is an $N_1$ such that for $n\geq N_1$,
$$
\Vert\mathcal L'_{n\rightarrow2n}\Vert<e^{n(\lambda_2+\epsilon)}\text{ and }\Vert\mathcal L_{n\rightarrow2n}\Vert>e^{n(\lambda_1-\epsilon)}.
$$
Clearly $\mathcal L'_{n\rightarrow2n}\neq\mathcal L_{n\rightarrow2n}$ since the latter has a greater norm.
Therefore there must be some $j$ such that
$$
\mathcal L'_{n\rightarrow2n}=\mathcal L'_{n\rightarrow j}\circ\Pi_{\sigma^j\omega}\circ\mathcal L'_{j\rightarrow2n}=\mathcal L_{n\rightarrow2n}\circ\Pi_{\sigma^n\omega}.
$$
In addition, there exists an $N_2$ such that for $n\geq N_2$,
$$
g_{\calL[n][\sigma^n\omega]}(E(\sigma^n\omega))\in(e^{n(\lambda_1-\epsilon)},e^{n(\lambda_1+\epsilon)}).
$$
As a final condition, there exists some $N_3\in\mathbb N$ such that for all $n\geq N_3$, $g_{\calL[n][\omega]}(E(\omega))\in(e^{n(\lambda_1-\epsilon)},e^{n(\lambda_1+\epsilon)})$.
Let $x\in\mathbb S_X$.
Putting these together, for all $n\geq\max\{N_1,N_2,N_3,\tfrac1\epsilon\}$,
\begin{align*}
\Vert P_{\sigma^n\omega}x\Vert\leq&e^{-n(\lambda_1-\epsilon)}\Vert\calL[n][\sigma^n\omega]P_{\sigma^n\omega}x\Vert\\
\leq&e^{-n(\lambda_1-\epsilon)}\big(\Vert\calL[n][\sigma^n\omega]x\Vert+\Vert\calL[n][\sigma^n\omega]\Pi_{\sigma^n\omega}x\Vert\big)\\
\leq&e^{-n(\lambda_1-\epsilon)}\big(e^{n(\lambda_1+\epsilon)}\Vert x\Vert+\Vert\mathcal L_{n\rightarrow 2n\omega}'x\Vert\big)\\
\leq& e^{-n(\lambda_1-\epsilon)}\big(e^{n(\lambda_1+\epsilon)}+\Vert{\mathcal L'}_{n\rightarrow2n}\Vert\big)\Vert x\Vert\\
\leq& e^{2n\epsilon}+e^{-n(\lambda_1-\lambda_2-2\epsilon)}\leq e^{3n\epsilon}.
\end{align*}
$\epsilon$ was arbitrary and the norms of $\Pi$ and $P$ differ by at most $1$ so $\tfrac1n\log\Vert P_{\sigma^n\omega}\Vert,\tfrac1n\log\Vert\Pi_{\sigma^n\omega}\Vert\rightarrow0$ as required.
\end{proof}
\begin{cor}
$V(\omega)=\tilde V(\omega)=\{x\in X:\lambda_\omega(x)\leq\lambda_2\}$.
\end{cor}
\begin{proof}
$\tfrac1n\log\Vert\calL[n][\omega]\vert_V(\omega)\Vert\rightarrow0$ establishes the fact that $V(\omega)\subseteq\tilde V(\omega)$.
Conversely, any $x\in\tilde V(\omega)\setminus V(\omega)$ would have $P_\omega(x)\neq 0$ so that $\lambda_\omega(x)=\lambda_1$, contradicting the definition of $\tilde V$.
\end{proof}
\section{Proof of main result}
Finally, we may conclude with the main result, a well behaved decomposition of the space acted on by a random linear dynamical system:
\begin{proof}[Proof of thorem \ref{mainres}]
The decomposition is obtained inductively.
At each stage it is shown that if $\lambda_{i+1}$ exists, there exists an equivariant decomposition
$$
X=E_{\leq i}(\omega)\oplus V_{i+1}(\omega)
$$
with $E_{\leq i}\in\mathcal{SM}(\omega\rightarrow\mathcal G_{M_i}X)$ and bounded projections
\begin{gather*}
\Pi_{i+1\omega}:X\rightarrow V_{i+1}(\omega)\text{ and}\\
P_{i\omega}:X\rightarrow E_{\leq i}(\omega).
\end{gather*}
The existence of the top fast space has already been established - here denote this $E_{<2}(\omega)\oplus V_2(\omega)$ with measurable projections $\Pi_{2\omega}$ and $P_{2\omega}$.

Suppose that the statement is true up to $i=l-1$.
If $\lambda_l=\nu$ then we are done, so suppose otherwise - that there exists $\lambda_{l+1}\geq\nu$.
The projection $\Pi_{l\omega}$ is pointwise bounded, so that there exists some $M>0$ such that $A=\{\Vert\Pi_{l\omega}\Vert<M\}$ has positive measure.
$$
\mathcal L_\omega'=\begin{cases}\mathcal L_\omega\circ\Pi_{l\omega},\omega\in A,\\ \mathcal L_\omega\text{ otherwise.}\end{cases}.
$$
As before $\mathcal L'$ is forward integrable.
Write $\lambda',\mu_i',\lambda_i',\nu',E_l\oplus V',M_i',P'$ and $\Pi'$ for the asymptotic growth rates, Lyapunov exponents, decomposition, fast space multiplicities, fast projection and slow projection with $\mathcal L'$.
There almost surely exists an $N\in\mathbb N$ such that for each $n\geq N$, the followng hold:
\begin{itemize}
\item $\calLp=\calL[n][\omega]\circ\Pi_{l\omega}=\Pi_{l\sigma^n\omega}\circ\calL[n][\omega]$.
\item
$\Vert\Pi_{\sigma^n\omega}\Vert<e^{n\epsilon}$
\item
$g_{\calL[n][\omega]}(E(\omega))>\Vert\calL[n][\omega]|_{V_{l+1}(\omega)}\Vert.$
\end{itemize}
Therefore applying lemma\ref{Tsmallgrow} and \ref{Tbiggrow} to $\calLp=\calL[n][\omega]\circ\Pi_\omega$, for each $k$ the following holds:
$$
\rho_{k+M_l}(\calL[n][\omega])\leq\rho_k(\calLp)\leq4\Vert\Pi_{\sigma^n\omega}\Vert\Vert\Pi_\omega\Vert\rho_{k+M_l}(\calL[n][\omega])\leq 4e^{n\epsilon}\Vert\Pi_\omega\Vert\rho_{k+M_l}(\calL[n][\omega]),
$$
whence $\mu_{M_l+k}=\mu_k'$ and $\nu=\nu'$.
Further then, $m_k'=m_{l+k}$ and $\lambda+k'=\lambda_{k+l}$.
Set $V_{l+1}(\omega)=V_l(\omega)\cap V'(\omega)$ and $\Pi_{l+1\omega}=\Pi_{l\omega}\circ\Pi'_\omega$.
Write $E_{\leq l}(\omega)=E_{<l}(\omega)+E_l(\omega)$.
The equality $\mathcal L_\omega=\mathcal L_\omega'$ holds on $V_l(\omega)$, whence $\lambda_\omega=\lambda_\omega'$ on $V_l(\omega)$.
Since $A$ has positive measure, there is almost surely an $N$ such that for all $n\geq N$,
$\calLp E_{<l}(\omega)=0,~\calLp[n][\sigma^{-n}\omega]X\subseteq V_l(\omega)$,
and $\calLp X\subseteq V_l(\sigma^n\omega).$
We may then conclude that for each $x\in X$, $\lambda_\omega'(x)\leq\lambda_l$ and $\lambda_1'\leq\lambda_l$.
As such, for each $i<l$, have $E_i(\omega)\subseteq V'(\omega)$.
As for $E_l(\omega)$, by equivariance $E_l(\omega)=\calLp[n][\sigma^{-n}\omega] E(\sigma^{-n}\omega)\subseteq V_l(\omega)$.
Thus $X=E_{<l}(\omega)\oplus E_l(\omega)\oplus(V_l(\omega)\cap V'(\omega))=E_{<l}(\omega)\oplus E_l(\omega)\oplus V_{l+1}(\omega)$.

$\Pi_{l+1}$ is then also tempered:
$$
0\leq\tfrac1n\log\Vert\Pi_{l+1\sigma^n\omega}\Vert\leq\tfrac1n\log\Vert\Pi_{l\sigma^n\omega}\Vert+\tfrac1n\log\Vert\Pi_{\sigma^n\omega}'\Vert\rightarrow0.
$$
Let $\epsilon>0$.
There exists an $N$ such that for $n\geq N,~\Vert P_{\sigma^n\omega}'\Vert,\Vert\Pi_{\sigma^n\omega}'\Vert<e^{n\epsilon}$ and so for all $x\in X$ we have $\max\{\Vert P_{\sigma^n\omega}'\calL[n][\omega]x\Vert,\Vert\Pi_{\sigma^n\omega}'\calL[n][\omega]x\Vert\}\leq\Vert\calL[n][\omega]x\Vert$.

Rearranging this last inequality and letting $x\in E_{\leq l}(\omega)$,
\begin{align*}
\Vert\calL[n][\omega]x\Vert\geq&e^{-n\epsilon}\max\{\Vert P_{\sigma^n\omega}'\calL[n][\omega]x\Vert,\Vert\Pi_{\sigma^n\omega}'\calL[n][\omega]x\Vert\}\\
\geq&e^{-n\epsilon}\max\{\Vert\calL[n][\omega] P_\omega'x\Vert,\Vert\calL[n][\omega]\Pi_\omega'x\Vert\}\\
\geq&e^{-n\epsilon}\max\{e^{n(\lambda_l-\epsilon)}\Vert P_{\omega}'x\Vert,e^{n(\lambda_1'-\epsilon)}\Vert\Pi_\omega'x\Vert\}\\
\geq&e^{-n\epsilon}e^{n(\lambda_l-\epsilon)}\max\{\Vert\Pi_\omega'x\Vert,\Vert P_\omega'x\Vert\}\geq\tfrac12e^{n(\lambda_l-2\epsilon)}.
\end{align*}
For $n\geq\max\{N,\tfrac1\epsilon\log2\}$ it follows that $g_{\calL[n][\omega]}(E_{\leq l}(\omega))\geq e^{n(\lambda_l-3\epsilon)}$.
The characterisation
$$
V_{l+1}(\omega)=\{x\in X:\limsup_{n\rightarrow\infty}\tfrac1n\log\Vert\calL[n][\omega]x\Vert\leq\lambda_{l+1}\}
$$
holds.

\end{proof}
\bibliography{proof_met}{}
\bibliographystyle{plain}
\appendix
\section*{Appendix: Equivalence of growth statistics}
The Gelfand numbers may be defined by $s_k(T)=\inf_{V\in\mathcal G^{k-1}X}\Vert T|_V\Vert$.
Throughout this section we assume $T\in\mathcal B(X)$.
\begin{lem}
$\rho_k\leq s_k$.
\end{lem}
\begin{proof}
Let $\epsilon>0$.
Choose $V\in\mathcal G_kX$ with $g_T(V)\geq\rho_k(T)-\epsilon$.
Then since any $W\in\mathcal G^{k-1}X$ intersects $V$ nontrivially, suppose $x$ is a unit vector in the intersection.
Then $\rho_k(T)\leq\Vert Tx\Vert\leq\Vert T|_W\Vert$.
$W$ was arbitrary so $\rho_k(T)\leq s_k(T)$.
\end{proof}
A better bound than the following may be found in the work of Pietsch on $s$-numbers\cite{pietsch1974}:
\begin{lem}
For all $k$, $s_k\leq4^{k-1}((k-1)!)^{\tfrac12}\rho_k$.
\end{lem}
\begin{proof}
For $k=1$ each quantity is just the norm of $T$, so equality holds.
Write $c_i=4^{1-i}((i-1)!)^{-\tfrac12}$.
Suppose that the proposition holds for each $l<k$ with $k\geq2$.
Choose $U\in\mathcal G_{k-1}X$ with $g_T(U)\geq(1-\epsilon)\rho_{k-1}(T)$, so that in particular
$$
g_T(U)\geq(1-\epsilon)c_{k-1}s_{k-1}(T)\geq(1-\epsilon)c_{k-1}s_k(T).
$$
Choose a complement $TU\oplus V=X$ with $\Pi=\Pi_{TU\Vert V}$ such that $\Vert\Pi\Vert\leq(k-1)^{\tfrac12}+\epsilon$.
Since $T^{-1}V\in\mathcal G^{k-1}X$, we may choose $x\in\mathbb S_{T^{-1}V}$ with $\Vert Tx\Vert\geq s_k(T)-\epsilon$.
Set $W=U\oplus\spn\{x\}\in\mathcal G_kX$ and let $a=b+tx\in\mathbb S_W$.
Applying $T$,
\begin{align*}
\Vert Ta\Vert\geq&\max\{\frac{\Vert Tb\Vert}{\Vert\Pi\Vert},\frac{\Vert tTx\Vert}{\Vert1-\Pi\Vert}\}\\
\geq&(1+\sqrt{k-1}+\epsilon)^{-1}\max\{(1-\epsilon)c_{k-1}s_k(T)\Vert b\Vert,s_k(T)(1-\epsilon)|t|\}\\
\geq&(2\sqrt{k-1}+\epsilon)^{-1}(1-\epsilon)c_{k-1}\max\{\Vert b\Vert,|t|\}\\
\geq&\frac{(1-\epsilon)c_{k-1}}{4(\sqrt{k-1}+\epsilon)}s_k(T).
\end{align*}
$a$ was arbitrary so the final line is a lower bound for $g_T(W)$.
$\epsilon$ was also arbitrary, so that $\rho_k(T)\geq\frac1{4c_{k-1}\sqrt{k-1}}s_k(T)=c_ks_k(T)$.
\end{proof}
\begin{lem}
The Gelfand numbers satisfy
$$
s_k(T)\geq k^{-\tfrac12}\Vert T\Vert_c.
$$
\end{lem}
\begin{proof}
Let $V\in\mathcal G^{k-1}X$.
We may by the theorem of Kadets in \cite{wojtaszczyk_1991} choose a complement $V\oplus W=X$ with $\Vert\Pi_{V\Vert W}\Vert\leq\sqrt{k-1}+\epsilon$.
Then $T\circ\Pi_{W\Vert V}$ is finite rank, so $\Vert T\Vert_c\leq\Vert T-T\circ\Pi_{W\Vert V}\Vert=\Vert T\circ\Pi_{V\Vert W}\Vert\leq\Vert T\vert_V\Vert\Vert\Pi_{V\Vert W}\Vert\leq\Vert T\vert_V\Vert(\sqrt{k-1}+\epsilon)$.
Taking the inf over such $V$ and letting $\epsilon\rightarrow 0$ we obtain the bound.
\end{proof}
\begin{lem}
The index defined at the start of this article agrees with the usual index of compactness
$$
\lim_{n\rightarrow\infty}\mu_n=\nu=\kappa=\lim_{n\rightarrow\infty}\tfrac1n\log\Vert\calL[n][\omega]\Vert_c.
$$
\end{lem}
\begin{proof}
Since $\rho$ dominates the compactness seminorm up to a multiplicative constant, $\kappa\leq\mu_k$ for every $k\in\mathbb N$, so certainly $\kappa\leq\nu$.
It remains to verify that $\kappa\geq\nu$.
There exists a $\delta>0$ such that for all $\mathbb P(\Lambda)<\delta,~\int_\Lambda\log\Vert\mathcal L\Vert<\tfrac\epsilon2$.
Choose $N$ sufficiently large that
$$
\mathbb P(\Vert\calL[N][\omega]\Vert_c\geq e^{N(\kappa+\epsilon)})<\tfrac12\delta.
$$
Choose $r$ sufficiently large that
$$
\mathbb P(G)=\mathbb P(\calL[N][\omega]\mathbb B_X\text{ may be covered by at most }e^{rN}~e^{N(\kappa+\tfrac12\epsilon)}\text{-balls})>1-\delta.
$$
Set $f_n(\omega)=\log\inf\{t>0:\calL[n][\omega]\mathbb B_X\text{ is covered by }e^{rn}~t\text{-balls}\}$.
In this case
\begin{align*}
f_{m+n}(\omega)\leq&\log\inf\{ab:\calL[n][\omega]\mathbb B_X\text{ is covered by }e^{rn}~a\text{-balls,}\text{ and }\calL[m][\sigma^n\omega]\mathbb B_X\text{ by }e^{rm}~b\text{-balls}\}\\
=&f_n(\omega)+f_m(\sigma^n\omega)
\end{align*}
so that the family is subadditive, $\int_\Omega f_1\leq\int_\Omega\log\Vert\mathcal L_\omega\Vert<\infty$ and
$$
\tfrac1N\int_\Omega f_N\leq\kappa+\tfrac12\epsilon+\tfrac1N\int_{G^c}\log\Vert\calL[N][\omega]\Vert<\kappa+\tfrac12\epsilon+\tfrac12\epsilon.
$$
Thus we may apply Kingman again to obtain that $\tfrac1nf_n(\omega)\rightarrow C<\kappa+\epsilon$.
Almost surely, for sufficiently large $n$ we may guarantee the following:
\begin{itemize}
\item $\Vert\calL[n][\omega]\Vert_c<e^{n(\kappa+\epsilon)}$
\item There is a $V\in\mathcal G_kX$ with $g_{\calL[n][\omega]}(V)>e^{n(\mu_k-\epsilon)}$.
\item $\calL[n][\omega]\mathbb B_X$ is covered by at most $e^{rn}$ $e^{n(\kappa+\epsilon)}$-balls
\item $e^{n\epsilon}>(2k)^k$.
\end{itemize}
Choose a basis of unit vectors $x_i$ for $V$ with $d(x_i,\spn\{x_j\}_{j<i})=1$.
Write
$$
\Lambda=\big\{\sum_{i=1}^k a_ix_i:a_i\in\{0,\pm\frac{2e^{n(\kappa+\epsilon)}}{\Vert\calL[n][\omega]x_i\Vert},\pm\frac{4e^{n(\kappa+\epsilon)}}{\Vert\calL[n][\omega]x_i\Vert},\cdots\},|a_i|<\tfrac1k\big\}\subseteq\mathbb B_X
$$
If two members $a=\sum_ia_ix_i$ and $b=\sum_ib_ix_i$ are distinct then there is a maximal $j\leq k$ with $a_j\neq b_j$.
Then
$$
\Vert a-b\Vert\geq |a_j-b_j|d(\calL[n][\omega]x_j,\calL[n][\omega]\spn\{x_i:i\leq j\})>e^{n(\kappa+\epsilon)}>\Vert\calL[n][\omega]\Vert_c.
$$
The points in $\calL[n][\omega]\Lambda$
are then of distance at least $2e^{n(\kappa+\epsilon)}>\Vert\calL[n][\omega]\Vert_c$ apart, and there are at least
$$
|\Lambda|=\prod_{i=1}^k\tfrac12\big\lfloor\tfrac{\Vert\calL[n][\omega]x_i\Vert}{e^{n(\kappa+\epsilon)}k}\big\rfloor\geq(2k)^{-k}\prod_{i=1}^ke^{n(\mu_k-\epsilon)-n(\kappa+\epsilon)}=e^{nk(\mu_k-\kappa-3\epsilon)}
$$
of them.
Each member of a cover of $\calL[n][\omega]\mathbb B_X$ by $e^{n(\kappa+\epsilon)}$-balls contains at most one element of $\Lambda$, so the cover has cardinality at least $e^{nk(\mu_k-\kappa-3\epsilon)}$.
On the other hand, this quantity is bounded by $e^{rn}$:
$$
e^{rn}\geq e^{nk(\mu_k-\kappa-4\epsilon)}.
$$
Taking $\log$ and rearranging we obtain
$$
k\leq\frac r{\mu_k-\kappa-4\epsilon}.
$$
In the case that $\mu_k>\kappa+5\epsilon$, it must hold that $k<\tfrac r\epsilon$.
This bound on $k$ shows the number of $\mu$s greater than $\kappa+5\epsilon$ is finite, whence $\mu_k\downarrow\kappa=\nu$ as $k\rightarrow\infty$.
\end{proof}
\end{document}